\documentclass[12pt,a4paper]{article}
\usepackage{amssymb, amscd, amsthm, amsmath,latexsym, amstext,mathrsfs}
\usepackage[all]{xy}

\newtheorem{thm}{Theorem}[section]
\newtheorem{coro}[thm]{Corollary}
\newtheorem{lemma}[thm]{Lemma}
\newtheorem{prop}[thm]{Proposition}

\theoremstyle{remark}
\newtheorem{remark}[thm]{\textbf{Remark}}

\theoremstyle{definition}

\newtheorem{notation}[thm]{Notation}

\numberwithin{equation}{thm}
\def\ker{\mathrm{Ker}}
\newcommand{\im}{\mathrm{Im}}

\newcommand{\Spec}{\mathrm{Spec}}

\newcommand{\lra}{\longrightarrow}
\newcommand{\simto}{\xrightarrow{\sim}}
\newcommand{\set}[1]{\{\,{#1}\,\}}

\newcommand{\ov}[1]{\overline{#1}}

\DeclareFontEncoding{OT2}{}{} 
  \newcommand{\textcyr}[1]{%
    {\fontencoding{OT2}\fontfamily{wncyr}\fontseries{m}\fontshape{n}%
     \selectfont #1}}
\newcommand{\sha}{{\mbox{\textcyr{Sh}}}}

\newcommand{\newpara}{\noindent\refstepcounter{thm}{\bf(\thethm)\;}}  

\begin{document}
\title{\textbf{A cohomological Hasse principle over two-dimensional local rings}}
\author{Yong HU}
\date{}

\maketitle

\begin{abstract}
Let $K$ be the fraction field of a two-dimensional henselian, excellent, equi-characteristic local domain. We prove a local-global principle for Galois cohomology with certain finite coefficients over $K$. We use classical machinery from \'etale cohomology theory, drawing upon an idea in Saito's work on two-dimensional local class field theory. This approach works equally well over the function field of a curve over an equi-characteristic henselian discrete valuation field, thereby giving a different proof of (a slightly generalized version of) a recent result of Harbater, Hartmann and Krashen. We also present two applications. One is the Hasse principle for torsors under quasi-split semisimple simply connected groups without $E_8$ factor. The other gives an explicit upper bound for the Pythagoras number of a Laurent series field in three variables. This bound is sharper than earlier estimates.
\end{abstract}


{\bf MSC2010 classes:} \ 11E72,   11E25, 19F15



\section{Introduction}

Let $F$ be a field and $n\ge 1$ an integer that is invertible in $F$. For an integer $i\ge 0$, let $\mathbb{Z}/n(i)$ be the tensor product of $i$ copies of $\mu_n$, where $\mu_n$ denotes the Galois module (or \'etale sheaf in a more general context) of $n$-th roots of unity over varying bases.
Let $\Omega_F$ denote the set of normalized discrete valuations on $F$ and for each $v\in \Omega_F$, denote by $F_v$ the completion of $F$ at $v$. If $F$ is a global function field (i.e. the function field of a curve over a finite field), then the classical Albert--Brauer--Hasse--Noether theorem for Brauer groups implies the Hasse principle for the cohomology group
$H^2(F,\,\mathbb{Z}/n(1))$, i.e., the injectivity of the natural map
\[
H^2(F,\,\mathbb{Z}/n(1))\lra \prod_{v\in\Omega_F}H^2(F_v,\,\mathbb{Z}/n(1))\,.
\]The same is true for a number field if we enlarge $\Omega_F$ by adjoining the real places of $F$.

In his influential work on higher dimensional class field theory \cite{Ka86}, Kato suggested that  higher dimensional analogs of this Hasse principle should be concerned with the cohomology groups $H^r(F,\,\mathbb{Z}/n(r-1))$ for $r>1$. In that paper he proved, among others, a theorem which implies the Hasse principle for $H^3(F,\,\mathbb{Z}/n(2))$ when $F$ is the function field of a curve over a $p$-adic field. At almost the  same time, Saito (\cite{Sai86}, \cite{Sai87}) studied local class field theory over two-dimensional local rings with finite residue fields. It can be shown that Saito's work yields an analog of Kato's theorem. Namely, if $F$ is the fraction field of a two-dimensional, henselian, excellent, local domain with finite residue field, then the Hasse principle for the group $H^3(F,\,\mathbb{Z}/n(2))$ holds (cf. \cite[Prop.$\;$4.1]{Hu11}). In recent years, Kato's and Saito's theorems have been relied on in a couple of papers to derive Hasse principles for torsors under semisimple groups over the relevant fields. We may cite for example \cite{CTPaSu},  \cite{Preeti13} and \cite{Hu12}. These results can be viewed as extensions of earlier work over two-dimensional geometric fields with algebraically closed residue fields  (cf. \cite{CTOP}, \cite{CTGP}, \cite{BKG04}).

Research interests in related problems have been summed up by Colliot-Th\'el\`ene \cite{CTLille11} to two
local-global questions, which aim to generalize the forementioned results by assuming no cohomological condition on the residue field.
As we will work over the same base field as in his questions, let us now fix the setup and notation.

\vskip2mm

\begin{notation}\label{notationCohHasse1p1v3} Let $K$ be a field of one of the following types:

(a) The function field of a connected regular projective curve over the fraction field of a henselian excellent discrete valuation ring $A$.  We call this case the \emph{semi-global} case.

(b) The fraction field of a two-dimensional henselian excellent normal local domain $A$. This will be called the \emph{local} case.

Let $\Omega_K$ denote the set of (normalized, rank 1) discrete valuations on $K$, and let $K_v$ be the corresponding completion of $K$ for each $v\in \Omega_K$.
In both cases we denote by $k$ the residue field of $A$. Let $n\ge 1$ be an integer \emph{that is invertible in} $k$.

By resolution of embedded singularities on two-dimensional schemes (\cite[p.38 and p.43]{Sha66}, \cite[p.193]{Lip75}), in both cases one can find
a connected regular two-dimensional scheme $\mathcal{X}$ equipped with a proper morphism  $\pi: \mathcal{X}\to\Spec(A)$ satisfying the following conditions:

\begin{itemize}
  \item The reduced closed fiber $Y$ of $\pi$ is a \emph{simple normal crossing} (snc) divisor on $\mathcal{X}$.
  \item The function field of $\mathcal{X}$ is isomorphic to $K$.
  \item In the semi-global case the morphism $\pi$ is flat, and in the local case $\pi$ is birational.
\end{itemize}

The use of such an $A$-scheme $\mathcal{X}$ will be necessary in the proofs of our main results and can make the statements a little stronger. So we will fix it as part of our notation.
\end{notation}

\vskip2mm

The two questions of Colliot-Th\'el\`ene are the following:

(1) Let $r\ge 2$ be an integer. Is the natural map
\[
H^r(K,\,\mathbb{Z}/n(r-1))\lra \prod_{v\in \Omega_K}H^r(K_v,\,\mathbb{Z}/n(r-1))
\]injective?

(2) Let $G$ be a smooth connected linear algebraic group over $K$. Does the natural map of pointed sets
\[
H^1(K,\,G)\lra \prod_{v\in\Omega_K}H^1(K_v,\,G)
\]have trivial kernel? (If the answer is ``yes'', this means that a $G$-torsor over $K$ has a $K$-rational point if and only if it has a $K_v$-point for each $v\in\Omega_K$.)

\vskip2mm

As mentioned above, quite a number of results on question (2) have been obtained when the residue field $k$ is assumed algebraically closed or finite. In the semi-global case, question (2) together with local-global problems in a different but closely related context has been studied in \cite{HHK} and \cite{HHK15AJM}. There a certain $K$-rationality assumption on the group $G$ plays a special role in the arguments. Counterexamples have been given in \cite{CTPS16TAMS}  to show that the answer to question (2) can be negative without the rationality assumption.

Nevertheless, if $G$ is a semisimple simply connected group, the Hasse principle for $G$-torsors is expected even when $G$ is not $K$-rational. For groups of many types, question (2) can be treated by using a positive answer to question (1) via cohomological invariants, as was done in most of the known results.

In the semi-global case, if $A$ is complete and equi-characteristic, the cohomological Hasse principle in question (1) has been proved by Harbater, Hartmann and Krashen (\cite[Thm.$\;$3.3.6]{HHK12}) using a patching method. This result makes no additional assumption on the residue field and hence yields Hasse principles for torsors under certain quasi-split groups in this generality (cf. \cite[$\S$4.3]{HHK12}).

\

In this paper, we give a positive answer to question (1) in the equi-characteristic local case and discuss some applications to question (2).

The goal is to prove the following

\begin{thm}\label{thmCohHasse1p2v3}\footnote{As V. Suresh has observed recently, if we consider discrete valuations coming from all possible models $\mathcal{X}$,  this theorem (and hence all its consequences) can also be proved in the mixed characteristic case at least when assuming the existence of $n$-th roots of unity in $K$. His argument combines a refinement of our method with analysis of the ramification of each cohomology class on a suitably chosen model $\mathcal{X}$, and will be written up separatedly. }
With notation as in Notation$\;\ref{notationCohHasse1p1v3}$, assume $A$ is equi-characteristic. We identify the set
$\mathcal{X}^{(1)}$ of codimension $1$ points of $\mathcal{X}$ with the subset of $\Omega_K$ consisting of discrete valuations defined by these points.

In both the semi-global case and the local case,  the natural map
\[
H^r(K\,,\,\mathbb{Z}/n(r-1))\lra \prod_{v\in\mathcal{X}^{(1)}}H^r(K_v,\,\mathbb{Z}/n(r-1))
\]
is injective for every $r\ge 2$. As a consequence, the same holds if the product on the right hand side is taken over all $v\in\Omega_K$.
\end{thm}

Note that no assumption (except for the restriction on the characteristic) is required on the residue field.

The $r=2$ case of the above theorem  was essentially observed in \cite{CTOP}, as was explained in \cite[$\S$3]{Hu10} and \cite[Thm.$\;$3.2.3]{HuThese}.

Grosso modo, the proof of Harbater, Hartmann and Krashen in the semi-global case goes as follows. They first observe that if $R$ is an equi-characteristic  \emph{regular} local ring with fraction field $F$, then the Hasse principle for $H^r(F,\,\mathbb{Z}/n(r-1))$ follows from Panin's work on Gersten's conjecture for equi-characteristic regular local rings \cite{Panin03}.
Then patching techniques are developed to reduce the cohomological Hasse principle in the semi-global case to the \emph{regular} local case.
To this end, they define ``patch fields'' from data given by the closed fiber of $\mathcal{X}$ over $A$, and prove exact sequences of Mayer--Vietoris type leading to local-global principles with respect to the patches.

Our proof also starts with the regular local case, which is built upon Panin's work, whence the equi-characteristic hypothesis in the theorem. As in \cite{HHK12}, the Bloch--Kato conjecture has to be used. The major difference between the two methods is that we do not employ any patching machinery but only \'etale cohomology theory. We distinguish discrete valuations centered in the closed fiber $Y$ and those centered in the open complement $U=\mathcal{X}\setminus Y$, and use the hypercohomology of a suitable \'etale complex that carries information from both $Y$ and $U$  (see $\S$\ref{sec2} for details). This latter idea, borrowed from Saito's work on two-dimensional local class field theory (\cite{Sai86}, \cite{Sai87}), is the key to the avoidance of the machinery of \cite{HHK12}. This approach enables us to treat the local case and the semi-global case in a parallel way, and in our opinion, it is likely that its combination with the patching method may lead to further observations in the future.

\

In the second part of the paper, we present applications of our main theorem to Hasse principles for torsors under semisimple simply connected groups. In particular, we will prove in section \ref{sec3} (after the proof of Theorem\;\ref{thmCohHasse3p5v3}) the following  theorem.

\begin{thm}\label{thmCohHasse1p3v3}
Let $K$ be as in Theorem$\;\ref{thmCohHasse1p2v3}$ and let $G$ be a quasi-split, semisimple, simply connected group without $E_8$ factor over $K$. Assume that the order of the Rost invariant of every simple factor of $G$ is not divisible by the characteristic $\mathrm{char}(K)$.

Then the natural map
\[
H^1(K,\,G)\lra \prod_{v\in\Omega_K}H^1(K_v,\,G)
\]
 has trivial kernel. If $G$ is absolutely simple over $K$, one can fix a model $\mathcal{X}$ as in Notation$\;\ref{notationCohHasse1p1v3}$ and replace the index set $\Omega_K$ with $\mathcal{X}^{(1)}$.
\end{thm}
Note that even in the semi-global case this theorem covers several groups that go beyond the results in \cite[$\S$4.3]{HHK12}. In the mixed characteristic case, if $A$ is a complete discrete valuation ring, the theorem is still true for quasi-split groups with only factors of classical types or of type $G_2$ (see Remark$\;$\ref{remarkCohHasse3p7v3}).

\

As another application of our cohomological Hasse principle, we will prove in section \ref{sec4} an upper bound for the Pythagoras number of Laurent series fields in three variables. As a notational convention, letters $x,\,y,\,z,\,t,\,t_1,\dotsc$ will be used to denote independent variables unless otherwise stated.

For any field $F$ of characteristic different from 2, the Pythagoras number $p(F)$ of $F$ is the smallest integer $p\ge 1$ or infinity, such that every sum of (finitely many) squares in $F$ can be written as a sum of $p$ squares.

The study of Pythagoras number for Laurent series fields in more than one variables was initiated by Choi, Dai, Lam and Reznick \cite{CDLR}. They obtained the equality $p(k(\!(x,\,y)\!))=2$ for a real closed field $k$ (e.g. $k=\mathbb{R}$). For an arbitrary field $k$, they proved that if $p(k(t))$ is bounded by a 2-power, then $p(k(\!(x,\,y)\!))$ is bounded by the same $2$-power. This result is strengthened in  \cite[Thm.\;1.1]{Hu13} by using Becher, Grimm and Van Geel's work \cite{BGvG12} on the Pythagoras number of fields of the shape $k(\!(x)\!)(t)$. (See also \cite{Sch01} for further results over general two-dimensional local rings and their fraction fields.)

It is conjectured  that the inequality $p(\mathbb{R}(\!(t_1,\dotsc, t_n)\!))\le 2^{n-1}$ is still true when $n\ge 3$ (\cite[p.80, Problem$\;$9]{CDLR}). The $n=3$ case is recently proved in \cite{Hu13}. For $n\ge 4$, the best upper bound until now is $2^n$. In fact, for any field $F$ such that $F(\sqrt{-1})$ has finite cohomological 2-dimension $r$, a theorem of Pfister combined with Milnor's conjecture (now proved in \cite{OVV07}) implies $p(F)\le 2^r$ (see section$\;$\ref{sec4}).

In this paper we will prove the following theorem.

\begin{thm}\label{thmCohHasse1p4v3}
  Let $k$ be a field of characteristic different from $2$ and let $r\ge 2$ be an integer. Then
  \[
  p(k(x,\,y))\le 2^r \quad \text{implies} \quad p(k(\!(x,\,y,\,z)\!))\le  p(k(\!(x,\,y)\!)(t))\le  p(k(\!(x,\,y)\!)(\!(z,\,t)\!))\le 2^r\,\]
   and
   \[
   p(k(x,\,y))=2^r \quad \text{implies} \quad  p(k(\!(x,\,y,\,z)\!))=p(k(\!(x,\,y)\!)(t))=p(k(\!(x,\,y)\!)(\!(z,\,t)\!))=2^r\,.
   \]
\end{thm}

There are two sample cases to which the theorem applies: $k$ is finitely generated over $\mathbb{R}$ or $\mathbb{Q}$. In both cases our theorem yields sharper bounds than Pfister's (cf. Corollary$\;$\ref{corCohHasse4p5v3}). Note also that the cohomological dimension of $k(\sqrt{-1})$ increases by $m$ when $k$ is replaced with an $m$-fold iterated Laurent series field $k_m=k(\!(t_1)\!)\cdots (\!(t_m)\!)$. So Pfister's bound for
$p(k_m(\!(x,\,y,\,z)\!))$ grows rapidly as $m$ goes to infinity. Our hypothesis $p(k(x,\,y))\le 2^r$ in  Theorem$\;\ref{thmCohHasse1p4v3}$ is insensible to this procedure. Therefore, we get the same bound for $p(k_m(\!(x,\,y,\,z)\!))$ (cf. Corollary$\;$\ref{coroCohHasse4p7v3}).

Together with earlier results in \cite{CDLR}, Theorem$\;\ref{thmCohHasse1p4v3}$ suggests that at least for $n\le 3$,  upper bound estimates
of the Pythagoras number of a Laurent series field in $n$ variables can be reduced to the case of a rational function field in $n-1$ variables. This is consistent in philosophy with Conjecture$\;$5.4 of \cite{Hu13}. (That conjecture combined with \cite[Conjecture\;4.15]{BGvG12} will imply $p(k(\!(x,\,y,\,z)\!))=p(k(x,\,y))$ in the general case.)

\section{Proof of the main theorem}\label{sec2}

Given a field $F$ and an integer $r\ge 0$, let $\mathrm{K}_r^M(F)$ denote the $r$-th Milnor $K$-group of $F$. For any integer $n$, we write
$\mathrm{K}^M_r(F)/n$ for $\mathrm{K}^M_r(F)\otimes(\mathbb{Z}/n\mathbb{Z})$. When $r\ge 1$, the group $\mathrm{K}^M_{r}(F)/n$ is
generated by symbol classes, i.e., elements of the form
\[
\bar{a}_1\otimes \cdots\otimes \bar{a}_{r}\,,
\]where $\bar{a}_i\in \mathrm{K}^M_1(F)/n=F^*/F^{*n}$.

The proof of our main theorem will use the following $K$-theoretic fact, which should be well known to experts.

\begin{lemma}\label{lemmaCohHasse2p1v3}
  Let $v_1,\dotsc, v_m$ be a finite collection of independent discrete valuations on a field $F$. Denote by $F_i$ the henselization of $F$ at $v_i$ for each $i$. Let $n\ge 1$ be an integer that is invertible in each of the residue fields $\kappa(v_i)$. Then for every $r\ge 1$, the natural map
\[
g_r\,:\;\; \mathrm{K}^M_{r}(F)/n\lra \bigoplus_{i}\mathrm{K}^M_{r}(F_i)/n
\]is surjective.
\end{lemma}
\begin{proof}
 We first prove the surjectivity of $g_1$.  Indeed,  by the weak approximation property for a finite number of discrete valuations, for any $(a_i)\in \bigoplus_iF_i^*$ there exists an element $a\in F$ such that $v_i(a-a_i)>v_i(a_i)$ for each $i$. Then $a/a_i$ is a unit for $v_i$ and its canonical image in $\kappa(v_i)$ is 1. By the invertibility assumption on $n$ and Hensel's lemma, we have $a=a_i\in \mathrm{K}_1^M(F_i)/n=F^*_i/F_i^{*n}$ for every $i$, whence the claimed surjectivity.

Now consider the general case. For any $\alpha=(\alpha_i)\in\bigoplus \mathrm{K}_{r}^M(F_i)/n$, one can choose an integer $N\ge 1$ possibly depending on $\alpha$, such that each $\alpha_i$ is the sum of $N$ symbols
\[
s_{i,\,j}\in \mathrm{K}_{r}^M(F_i)/n\,,\qquad 1\le j\le N\,.
\]By the surjectivity of $g_1$, there exist symbols $s_1,\dotsc, s_N\in \mathrm{K}^M_{r}(F)/n$   such that for every $i$ and every $j$, the image of $s_j$ in $\mathrm{K}_{r}^M(F_i)/n$ coincides with $s_{i,\,j}$. Thus, the map
$g_r$ sends $a:=\sum_js_j$ to $\alpha$, showing that $g_{r}$ is surjective.
\end{proof}

As already observed in \cite[Prop.$\;$3.3.4]{HHK12}, the following lemma is an easy consequence of Panin's work  \cite{Panin03} on Gersten's conjecture for equi-characteristic regular local rings.

\begin{lemma}\label{lemmaCohHasse2p2v3}
Let $R$ be an equi-characteristic henselian regular local ring with fraction field $F$ and  $m>0$ an integer invertible in $R$. Let $R^{(1)}$ denote the set of discrete valuations of $F$ corresponding to codimension $1$ points of $\Spec(R)$. For each $v\in R^{(1)}$, denote by $F_{(v)}$ the henselization of $F$ at $v$ and $\kappa(v)$ the residue field of $v$.
Let $\pi\in R$ be a regular parameter in $R\,($i.e. $\pi$ is an element of a regular system of parameters of the regular local ring $R)$, identified with the element of $R^{(1)}$ defined by the prime ideal $\pi R$ of $R$.

Then for any integers $r,\,j\ge 1$, the natural map
\[
(\rho,\,\partial)\,: \;\; H^r(F\,,\,\mathbb{Z}/m(j))\lra H^r(F_{(\pi)}\,,\,\mathbb{Z}/m(j))\oplus \bigoplus_{v\in R^{(1)}\,,\,v\neq \pi} H^{r-1}(\kappa(v)\,,\,\mathbb{Z}/m(j-1))
\]is injective, where $\rho: H^r(F,\,\mathbb{Z}/m(j))\to H^r(F_{(\pi)}\,,\,\mathbb{Z}/m(j))$ is the natural map induced by the inclusion $F\subseteq F_{(\pi)}$ and
\[
\partial\,:\; H^r(F\,,\,\mathbb{Z}/m(j))\lra  \bigoplus_{v\in R^{(1)}\,,\,v\neq \pi} H^{r-1}(\kappa(v)\,,\,\mathbb{Z}/m(j-1))
\]is induced by the residue maps.
\end{lemma}
\begin{proof}
Put $M=\mathbb{Z}/m(j)$. Let $\alpha\in H^r(F\,,\,M)$ be such that $\rho(\alpha)=0=\partial(\alpha)$. By \cite[$\S$5, Thm.$\;$C]{Panin03}, we have an exact sequence of \'etale cohomology groups
\[
0\lra H^r(R\,,\,M)\overset{\iota}{\lra} H^r(F\,,\,M)\lra\bigoplus_{v\in R^{(1)}}H^{r-1}(\kappa(v)\,,\,M(-1))\,.
\]For each $v\in R^{(1)}$, the map $H^r(F,\,M)\to H^{r-1}(\kappa(v)\,,\,M(-1))$ factors through the natural map
$H^r(F\,,\,M)\to H^r(F_{(v)}\,,\,M)$. So $\alpha$ maps to 0 in each $H^{r-1}(\kappa(v)\,,\,M(-1))$ and thus comes from an element $\tilde{\alpha}\in H^r(R\,,\,M)$.

Denote by $R^h_{\pi}$ the henselization of $R$ at $\pi$. Its residue field $\kappa(\pi)$ is the fraction field of the quotient ring $\ov{R}_{\pi}=R/(\pi)$, which is a henselian regular local ring with the same residue field as $R$.  (Here of course $\overline{R}_{\pi}$ is  a field itself if $\dim R=1$. In this case we use the convention that a field is a henselian local ring.) Panin's theorem (\cite[$\S$5, Thm.$\;$C]{Panin03}) applied to $\ov{R}_{\pi}$ implies in particular that the natural map $\iota_{\pi}\,:\;H^r(\ov{R}_{\pi}\,,\,M)\to H^r(\kappa(\pi)\,,\,M)$ is injective. On the other hand, by the proper base change theorem we have natural identifications
\[
H^r(R\,,\,M)=H^r(k\,,\,M)=H^r(\ov{R}_{\pi}\,,\,M)\,,
\]where $k$ denotes the residue field of $R$. Now consider the following two commutative diagrams
\[
\begin{CD}
H^r(R,\,M) @>{\varphi_{\pi}}>> H^r(R^h_{\pi}\,,\,M)\\
@V{\cong} VV @VVV\\
H^r(\ov{R}_{\pi}\,,\,M) @>{\iota_{\pi}}>> H^r(\kappa(\pi)\,,\,M)
\end{CD}\quad\text{ and }\quad
\begin{CD}
H^r(R,\,M) @>{\varphi_{\pi}}>> H^r(R^h_{\pi}\,,\,M)\\
@V{\iota} VV @VV{\tau_{\pi}}V\\
H^r(F\,,\,M) @>{\phi_{\pi}}>> H^r(F_{(\pi)}\,,\,M)
\end{CD}
\]From the left diagram we see that $\varphi_{\pi}$ is injective, and the right diagram shows that
\[
0=\phi_{\pi}(\alpha)=\phi_{\pi}(\iota(\tilde{\alpha}))=\tau_{\pi}(\varphi_{\pi}(\tilde{\alpha}))\,.
\]But we know that $\tau_{\pi}$ is injective (by the discrete valuation ring case of Gersten's conjecture, or by Panin's theorem cited above). So we get $\tilde{\alpha}=0$ and hence $\alpha=0$.
\end{proof}

\newpara\label{paraCohHasse2p3v3}
Recall some  important facts about hypercohomology, as will be used in our proof of Theorem$\;$\ref{thmCohHasse1p2v3}.

Let $X$ be a noetherian scheme. Let $D(X)$ denote the derived category of complexes of \'etale sheaves of abelian groups on $X$, and let $D_+(X)\subseteq D(X)$ be the full subcategory of complexes that are bounded below. The global section functor on \'etale sheaves admits a derived functor $R\Gamma: D_+(X)\to D_+(\mathbf{Ab})$, where $D_+(\mathbf{Ab})$ denotes the derived category of complexes bounded below of abelian groups. For a complex $\mathscr{F}\in D_+(X)$, the \emph{hypercohomology groups} of $\mathscr{F}$, denoted $H^i(X,\,\mathscr{F})$, are defined as the cohomology groups of the complex $R\Gamma(\mathscr{F})$, i.e.,
\[
H^i(X,\,\mathscr{F}):=H^i(R\Gamma(\mathscr{F}))\,,\quad \forall\; i\in \mathbb{Z}\,.
\]If the complex $\mathscr{F}$  is given by a sheaf concentrated in degree 0, then the hypercohomology groups are the usual cohomology groups.

Let $j: V\hookrightarrow X$ be an open immersion. The direct image functor $j_*$ has a derived functor $Rj_*: D_+(V)\to D_+(X)$. Since $j_*$ sends injectives to injectives, one has (by \cite[p.308, Chap. C.D., $\S$2.3, Prop.$\;$3.1]{SGA4.5})
$R(\Gamma\circ j_*)=R\Gamma\circ Rj_*$. Thus, for a complex $\mathscr{M}\in D_+(V)$,
\begin{equation}\label{eqCohHasse2p3p1v3}
H^i(X,\,Rj_*\mathscr{M})=H^i(R\Gamma\circ Rj_*(\mathscr{M}))=H^i(R(\Gamma\circ j_*)\mathscr{M})=H^i(V,\,\mathscr{M})\,
\end{equation}for all $i$.

Let $Z\subseteq X$ be a closed subscheme. For any sheaf $\mathscr{M}$ on $X$, put
\[
\Gamma_Z(\mathscr{M}):=\ker(\Gamma(X,\,\mathscr{M})\longrightarrow \Gamma(X\setminus Z\,,\,\mathscr{M}))\,.
\]This functor admits a derived functor $R\Gamma_Z: D_+(X)\to D_+(\mathbf{Ab})$. One defines the \emph{hypercohomology groups with support} in $Z$ of a complex $\mathscr{F}\in D_+(X)$ by
\[
H^i_Z(X,\,\mathscr{F}):=H^i(R\Gamma_Z(\mathscr{F}))\,.
\]For the same reason as above, for any $\mathscr{M}\in D_+(V)$ one has
\begin{equation}\label{eqCohHasse2p3p2v3}
  H^i_Z(X,\,Rj_*\mathscr{M})=H^i_{Z\cap V}(V,\,\mathscr{M})
\end{equation}for all $i\in \mathbb{Z}$.

\begin{notation}\label{notationCohHasse2p4v3}

Let $A,\,\mathcal{X}$ and so on be as in Notation\;\ref{notationCohHasse1p1v3}. To prove the main theorem we introduce some more notation:
\begin{enumerate}
  \item[(1)] $U:=\mathcal{X}\setminus Y$, where $Y\subseteq \mathcal{X}$ is the reduced closed fiber of $\pi: \mathcal{X}\to \Spec(A)$. The natural inclusion $j: U\hookrightarrow \mathcal{X}$ is the complement
of the closed immersion $i: Y\to \mathcal{X}$. In the semi-global case, $U$ is the generic fiber of $\pi$. In the local case, $\pi$ induces an isomorphism from $U$ to $\Spec(A)\setminus\{[\mathfrak{m}_A]\}$, where $[\mathfrak{m}_A]$ denotes the closed point of $\Spec(A)$.

  \item[(2)]  $Y_0$ denotes the set of closed points of $Y$ and $Y_1$ denotes the set of generic points (of the irreducible components) of $Y$.

  \item[(3)] $P:=\mathcal{X}^{(1)}\setminus Y_1$, the set of codimension 1 points of $\mathcal{X}$ outside $Y$.  One can identify $P$ with the set of closed points of $U$.  In the local case, $\pi$ maps $P$ bijectively onto $(\Spec(A))^{(1)}$.

 \item[(4)] For each $p\in P$, let $\bar p$ denote the scheme-theoretic closure of $p$ in $\mathcal{X}$. Then $\bar{p}\cong \Spec(B)$ for some henselian local domain $B$ with fraction field $\mathrm{Frac}(B)=\kappa(p)$, and therefore, $\bar p$ meets $Y$ at one and only one point (which is the unique closed point of $\bar{p}$).

 To see this, first note that $\bar{p}$ is an integral scheme.

 In the semi-global case,  the structural morphism $\pi: \mathcal{X}\to\Spec(A)$ induces a finite morphism $\bar{p}\to \Spec(A)$. Hence $\bar{p}$ is affine. Since $A$ is henselian and local, any integral finite $A$-algebra is henselian and local.

 In the local case, $p$ can be viewed as a height 1 prime ideal of $A$ and $\pi$ induces a finite birational morphism  $\bar{p}\to \Spec(A/p)$. The quotient $A/p$ is a henselian local domain, so the  assertion follows by the same argument as above.

  \item[(5)]
 For $x\in \mathcal{X}$, we denote by $A_{(x)}$ the henselization of the regular local ring $\mathscr{O}_{\mathcal{X},\,x}$ and $K_{(x)}:=\mathrm{Frac}(A_{(x)})$ the fraction field of $A_{(x)}$.

  \item[(6)] For any $x\in Y_0\subseteq\mathcal{X}$, put  $P_x:=\set{p\in P\,|\, x\in \bar p}$.

  \item[(7)] For a fixed point $x\in Y_0\subseteq \mathcal{X}$, let $\varphi=\varphi_x: \Spec(A_{(x)})\to \mathcal{X}$ be the canonical morphism. We say a codimension 1 point $p_x$ of $\Spec(A_{(x)})$ is \emph{vertical} if  $\varphi(p_x)\in Y_1$, i.e., if the prime ideal $p_x\cap \mathscr{O}_{\mathcal{X},\,x}$ of $\mathscr{O}_{\mathcal{X},\,x}$ corresponds to an irreducible component of the reduced closed fiber $Y$. Otherwise we say $p_x\in (\Spec(A_{(x)}))^{(1)}$ is \emph{horizontal}, namely, $p_x$ is horizontal if $\varphi(p_x)\in P$. We denote by $V_x$ (resp. $H_x$) the set of vertical (resp. horizontal) points of $\Spec(A_{(x)})$.

Alternatively, vertical and horizontal points can be described as follows:

Let $R\mapsto R^h$ denote the henselization functor of local rings.  If $I\subseteq \mathscr{O}_{\mathcal{X},\,x}$ is the ideal defining $Y$, then
\[
\mathscr{O}_{Y,\,x}^h=(\mathscr{O}_{\mathcal{X},\,x}/I)^h=\mathscr{O}_{\mathcal{X},\,x}^h/I\mathscr{O}_{\mathcal{X},\,x}^h=A_{(x)}\otimes_{\mathscr{O}_{\mathcal{X},\,x}}\mathscr{O}_{Y,\,x}\,
\]by \cite[IV.18.6.8]{EGAivPIHES32}. So we have a cartesian diagram
\[
\begin{CD}
\Spec(\mathscr{O}_{Y,\,x}^h) @>>> \Spec(A_{(x)})\\
@VVV @VV{\varphi}V\\
Y @>>> \mathcal{X}
\end{CD}
\]
If $x$ is not a regular point of $Y$, $\Spec(\mathscr{O}_{Y,\,x}^h)$ is not integral. Note however that, since $Y$ is an snc divisor on the regular scheme $\mathcal{X}$, each irreducible component of $\Spec(\mathscr{O}_{Y,\,x}^h)$ is defined by a \emph{regular parameter} in the regular local ring $A_{(x)}=\mathscr{O}_{\mathcal{X},\,x}^h$.
Therefore,  the cardinality of $V_x$  is 1 or 2, depending on whether $x$ is a regular point of $Y$ or not.
We may thus identify   $V_x$ with $(\Spec(\mathscr{O}_{Y,\,x}^h))^{(0)}$, the set of generic points of $\Spec(\mathscr{O}_{Y,\,x}^h)$, which is mapped bijectively to
$(\Spec(\mathscr{O}_{Y,\,x}))^{(0)}$ via $\varphi$.

Now assume  $p\in P_x$, i.e., $x\in \bar p$. The affine coordinate ring $B$ of the scheme $\bar p$ is a henselian local domain. Considering $p$ as a prime ideal of $\mathscr{O}_{\mathcal{X},\,x}$, we have
\[
A_{(x)}/pA_{(x)}=\mathscr{O}_{\mathcal{X},\,x}^h/p.\mathscr{O}_{\mathcal{X},\,x}^h\cong (\mathscr{O}_{\mathcal{X},\,x}/p)^h=B^h=B\,.\]It follows that $A_{(x)}$ has a unique prime ideal lying over $p$, which is given by $p.A_{(x)}$. So $\varphi$ induces a bijection
\[
H_x:=\varphi^{-1}(P_x)\simto P_x\,.
\]

Now consider $\varphi^{-1}(U)=\Spec(A_{(x)})\times_{\mathcal{X}}U$. Since the complement $\varphi^{-1}(Y)=\Spec(\mathscr{O}_{Y,\,x}^h)$ is a principal divisor ($A_{(x)}$ being a regular local ring), $\varphi^{-1}(U)$ is an affine scheme. We denote its affine ring by $R_x$, so that there is a cartesian diagram
\[
\begin{CD}
\Spec(R_x) @>>> \Spec(A_{(x)})\\
@VVV @VV{\varphi}V\\
U @>>> \mathcal{X}
\end{CD}
\]The set $H_x$ can also be viewed as the set of closed points of the scheme $\varphi^{-1}(U)=\Spec(R_x)$.
\end{enumerate}
\end{notation}

We shall now start the proof of Theorem$\;$\ref{thmCohHasse1p2v3}.

First note that we may replace each completion $K_v$ by the corresponding henselization $K_{(v)}$. Indeed, for every $v\in \mathcal{X}^{(1)}$, $K_v$ is a separable extension of $K_{(v)}$, and the natural map between the Galois cohomology groups induced by this extension is injective (see e.g.
\cite[p.624, $\S$1.5, Lemma$\;$12]{KatoII80}).

With notation as above, we have $\mathcal{X}^{(1)}=Y_1\cup P$.   The kernel of the local-global map
\[
H^r(K,\,\mathbb{Z}/n(r-1))\lra \prod_{v\in \mathcal{X}^{(1)}}H^r(K_{(v)}\,,\,\mathbb{Z}/n(r-1))
\]is  contained in the kernel of the map
{\small\[
(\rho\,,\,\partial):\;\; H^r(K,\,\mathbb{Z}/n(r-1))\lra\bigoplus_{\eta\in Y_1}H^r(K_{(\eta)}\,,\,\mathbb{Z}/n(r-1))\oplus\bigoplus_{p\in P}H^{r-1}(\kappa(p),\,\mathbb{Z}/n(r-2))\,,
\]}where the map $\rho$ is induced by the natural maps
\[
H^r(K,\,\mathbb{Z}/n(r-1))\to H^r(K_{(\eta)}\,,\,\mathbb{Z}/n(r-1))
\] and the map $\partial$ is induced by the residue maps
\[
\partial_{p}\,: H^r(K,\,\mathbb{Z}/n(r-1))\lra H^{r-1}(\kappa(p)\,,\,\mathbb{Z}/n(r-2))\,.
\]
So Theorem$\;$\ref{thmCohHasse1p2v3} follows immediately from the following:

\begin{thm}\label{thmCohHasse2p5v3}
With notation and hypotheses as in Theorem$\;\ref{thmCohHasse1p2v3}$, the map
{\small\[
(\rho\,,\,\partial):\;\; H^r(K,\,\mathbb{Z}/n(r-1))\lra\bigoplus_{\eta\in Y_1}H^r(K_{(\eta)}\,,\,\mathbb{Z}/n(r-1))\oplus\bigoplus_{p\in P}H^{r-1}(\kappa(p),\,\mathbb{Z}/n(r-2))\,
\]}
is injective.
\end{thm}
\begin{proof}Write $\Lambda=\mathbb{Z}/n(r-1)$. We may consider $P$ as the set of closed points of the scheme $U=\mathcal{X}\setminus Y$ (Notation\;\ref{notationCohHasse2p4v3} (3)), and we have the localization sequence of \'etale cohomology on $U$:
\begin{equation}\label{eqCohHasse2p5p1v3}
\bigoplus_{p\in P}H^r_p(U,\,\Lambda)\lra H^r(U,\,\Lambda)\lra H^r(K,\,\Lambda)\lra\bigoplus_{p\in P}H^{r+1}_p(U,\,\Lambda)
\end{equation}
By the absolute purity for discrete valuation rings (see e.g. \cite[p.139, Chap. Cycle, Prop.$\;$2.1.4]{SGA4.5}), we have
\[
H^d_{p}(U,\,\Lambda)=H^{d-2}(\kappa(p)\,,\,\Lambda(-1))\,,\quad \forall\; d\ge 0\,.
\]
Using this identification, the map $H^r(K,\,\Lambda)\to H^{r+1}_p(U,\,\Lambda)$ in \eqref{eqCohHasse2p5p1v3} can be identified with the residue map
$H^r(K,\,\Lambda)\to H^{r-1}(\kappa(p),\,\Lambda(-1))$ (cf. \cite[Lemma\;1.4 (2)]{Ka86}).

With natural inclusions $j: U\hookrightarrow \mathcal{X}$ and $i: Y\hookrightarrow\mathcal{X}$ as in Notation\;\ref{notationCohHasse2p4v3} (1), put $\mathscr{F}=i^*Rj_*\Lambda$. To avoid the need for patching techniques, we borrow from Saito's work \cite{Sai86} and \cite{Sai87} the idea of using the hypercohomology of $\mathscr{F}$. This gives rise to a commutative diagram with exact rows
{\small
\[
\xymatrix{
\bigoplus_{p\in P}H^r_{\bar p}(\mathcal{X},\,Rj_*\Lambda) \ar[d]_{} \ar[r]^{} & H^r(\mathcal{X},\,Rj_*\Lambda)\ar[d]_{} \ar[r]^{} & H^r(K,\,\Lambda) \ar[d]_{}
\ar[r]^{} & \bigoplus_{p\in P}H^{r+1}_{\bar p}(\mathcal{X},\,Rj_*\Lambda) \ar[d]_{} \\
\bigoplus_{x\in Y_0}H^r_{x}(Y,\,\mathscr{F}) \ar[r]^{} & H^r(Y,\,\mathscr{F})\ar[r]^{} & \bigoplus_{\eta\in Y_1}H^r(\kappa(\eta)\,,\,\mathscr{F})
\ar[r]^{} &\bigoplus_{x\in Y_0}H^{r+1}_{x}(Y,\,\mathscr{F})
}
\]
}where for each $p\in P$, $\bar p$ denotes its closure in $\mathcal{X}$. In this diagram, the second vertical arrow is an isomorphism by the proper base change theorem (cf. \cite[Exp.$\;$XII, Coro.$\;$5.5]{SGA4iii}).  The first row may be identified with the exact sequence in \eqref{eqCohHasse2p5p1v3}, by the functoriality of the functor $Rj_*$ (cf. \eqref{eqCohHasse2p3p1v3} and \eqref{eqCohHasse2p3p2v3}). For each $\eta\in Y_1$, there are canonical isomorphisms
\begin{equation}\label{eqcohHasse2p5p2new}
H^d(\kappa(\eta)\,,\,\mathscr{F})\cong H^d(K_{(\eta)}\,,\,\Lambda)\,,\quad\forall\; d\ge 0\,.
\end{equation}
In fact, $\eta\in Y_1$ can be considered as a codimension 1 point of $\mathcal{X}$. The henselization $\mathscr{O}_{\mathcal{X},\,\eta}^h$ of the local ring $\mathscr{O}_{\mathcal{X},\,\eta}$ is a henselian discrete valuation ring, and its fraction field is denoted by $K_{(\eta)}$. The proper base change theorem applied to the closed immersion $\mathrm{Spec}(\kappa(\eta))\to\mathrm{Spec}(\mathscr{O}_{\mathcal{X},\,\eta}^h)$ gives the isomorphism
\[
H^d(\kappa(\eta),\,\mathscr{F})\cong H^d(\mathscr{O}_{\mathcal{X},\,\eta}^h\,,\,Rj_*\Lambda)\;.
\]The group on the right hand side is isomorphic to $H^d(K_{(\eta)},\,\Lambda)$ by \eqref{eqCohHasse2p3p1v3}.

Putting all these together, we obtain a commutative diagram with exact rows
{\footnotesize
\begin{equation}\label{eqCohHasse2p5p2v3}
\xymatrix{
\bigoplus_{p\in P}H^{r-2}(\kappa(p),\,\Lambda(-1)) \ar[d]_{\gamma} \ar[r]^{\qquad\quad\alpha} & H^r(U,\,\Lambda)\ar[d]_{\cong} \ar[r]^{\phi} & H^r(K,\,\Lambda) \ar[d]_{\rho}
\ar[r]^{\partial\qquad} & \bigoplus_{p\in P}H^{r-1}(\kappa(p),\,\Lambda(-1)) \ar[d]_{} \\
\bigoplus_{x\in Y_0}H^r_{x}(Y,\,\mathscr{F}) \ar[r]^{\beta} & H^r(Y,\,\mathscr{F})\ar[r]^{\tau\qquad} & \bigoplus_{\eta\in Y_1}H^r(K_{(\eta)}\,,\,\Lambda)
\ar[r]^{\theta} &\bigoplus_{x\in Y_0}H^{r+1}_{x}(Y,\,\mathscr{F})
}
\end{equation}}

Here the maps $\rho$ and $\partial$ are concerned  with  (discrete valuations with center in) the closed fiber and its open complement respectively. We shall prove the desired injectivity by relating these maps to maps of more local nature via the above diagram. This method is the substitute for the Mayer--Vietoris machinery of \cite{HHK12} (e.g. \cite[Coro.$\;$3.1.6]{HHK12}). In fact, we need only to prove that the map $\rho$ is injective when restricted to $\ker(\partial)=\im(\phi)$. By an easy diagram chase, it suffices to show the surjectivity of the map
\[
\gamma\,:\; \bigoplus_{p\in P}H^{r-2}(\kappa(p),\,\Lambda(-1))=\bigoplus_{p\in P}H^r_{\bar p}(\mathcal{X},\,Rj_*\Lambda)\lra\bigoplus_{x\in Y_0}H^r_{x}(Y,\,\mathscr{F})\,.
\]
Here for each pair $(p,\,x)\in P\times Y_0$, the $(p,\,x)$-component $\gamma_{p,\,x}$ of $\gamma$ is 0 if $x$ does not lie in the closure $\bar p$ of $p$ in $\mathcal{X}$. Thus the map $\gamma$ decomposes into a direct sum
\[
\gamma=\bigoplus_{x\in Y_0}\left(\gamma_x:=\sum\gamma_{p,\,x}:\; \bigoplus_{p\in P_x}H^{r-2}(\kappa(p),\,\Lambda(-1))\lra H^r_{x}(Y,\,\mathscr{F})\right)\,.
\](Recall that $P_x=\set{p\in P\,|\, x\in \bar p}$.)
Therefore, the theorem will follow from the surjectivity of $\gamma_{x}$ for every $x\in Y_0$.

Notice that by considering the localization sequences on $\Spec(R_x)=\Spec(A_{(x)})\times_{\mathcal{X}}U$ and on $\Spec(\mathscr{O}^h_{Y,\,x})=\Spec(A_{(x)})\times_{\mathcal{X}}Y$ (cf. Notation\;\ref{notationCohHasse2p4v3} (7)), we have as an analog of \eqref{eqCohHasse2p5p2v3} the following commutative diagram with exact rows
{\small
\begin{equation}\label{eqCohHasse2p5p3v3}
\begin{split}
&\xymatrix{
H^{r-1}(K_{(x)},\,\Lambda) \ar[r]_{} \ar[d]_{g} & \bigoplus_{p\in H_x}H^{r-2}(\kappa(p),\,\Lambda(-1)) \ar[d]_{\gamma_x} \ar[r]^{\qquad\quad\alpha_x} & H^r(R_x,\,\Lambda)\ar[d]_{\cong} \ar[r]^{\qquad\phi_x} & \\
\bigoplus_{\eta\in V_x}H^{r-1}(K_{(x,\,\eta)}\,,\,\Lambda) \ar[r]^{} & H^r_{x}(Y,\,\mathscr{F}) \ar[r]^{\beta_x} & H^r(\mathscr{O}_{Y,\,x}^h,\,\mathscr{F})\ar[r]^{\qquad\quad\tau_x} &
}\\
&\xymatrix{
\ar[r]^{\phi_x\qquad} & H^r(K_{(x)},\,\Lambda) \ar[d]_{\rho_x}
\ar[r]^{\partial_x\qquad} &  \bigoplus_{p\in H_x}H^{r-1}(\kappa(p),\,\Lambda(-1)) \ar[d]_{} \ar[r]_{} &  \cdots \\
\ar[r]^{\tau_x\qquad\qquad} & \bigoplus_{\eta\in V_x}H^r(K_{(x,\,\eta)}\,,\,\Lambda) \ar[r]^{\theta_x} & H^{r+1}_{x}(Y,\,\mathscr{F}) \ar[r]^{} & \cdots
}\end{split}
\end{equation}
}($H_x$ and $V_x$ are the sets of horizontal and vertical points respectively, as defined in Notation\;\ref{notationCohHasse2p4v3} (7).) Here for each $\eta\in V_x$, $K_{(x,\,\eta)}$ is the henselization of $K_{(x)}=\mathrm{Frac}(A_{(x)})$ with respect to the discrete valuation corresponding to $\eta$. More precisely, the field $K_{(x,\,\eta)}$ is defined as follows: We localize $A_{(x)}$ at $\eta$ (considered as a codimension 1 point of $\mathrm{Spec}(A_{(x)})$) and let $A_{(x,\,\eta)}$ be the henselization of the localization. Then denote by $K_{(x,\,\eta)}$ the fraction field of $A_{(x,\,\eta)}$.

Note that in the bottom row of \eqref{eqCohHasse2p5p3v3}, the group $H^r(\mathscr{O}^h_{Y,\,x},\,\mathscr{F})$ plays the role of the group $H^r(Y,\,\mathscr{F})$ in \eqref{eqCohHasse2p5p2v3}, and we have $H^r_x(Y,\,\mathscr{F})=H^r_x(\mathscr{O}_{Y,\,x}^h,\,\mathscr{F})$ by \cite[p.93, Cor.\;III.1.28]{Mil80}. (So the map $\beta_x$ in \eqref{eqCohHasse2p5p3v3} is the local variant of the map $\beta$ in \eqref{eqCohHasse2p5p2v3}. And similarly for the maps $\tau_x$ and $\theta_x$.) That the cohomology of $\mathscr{F}$ on $\kappa(\eta)$ can be replaced by the cohomology of $\Lambda$ on $K_{(x,\,\eta)}$ is exactly an analog of the  isomorphism in \eqref{eqcohHasse2p5p2new}.

\

By the Bloch-Kato conjecture, which becomes a theorem thanks to the work of Rost, Voevodsky et al (see \cite{Voe11BKconj}), for any field $F$ of characteristic not dividing $n$, one has a canonical isomorphism
\[
H^{r-1}(F\,,\,\Lambda)=H^{r-1}(F,\,\mathbb{Z}/n(r-1))\cong \mathrm{K}^M_{r-1}(F)/n\,.
\]Since $r\ge 2$, it follows from Lemma\;\ref{lemmaCohHasse2p1v3} that the map $g$ in diagram \eqref{eqCohHasse2p5p3v3} is surjective (keeping in mind that the invertibility condition on $n$ is contained in Notation\;\ref{notationCohHasse1p1v3}).

On the other hand, since $Y$ is an snc divisor on $\mathcal{X}$, each of its irreducible components is locally defined by a regular parameter. So we may apply
Lemma$\;$\ref{lemmaCohHasse2p2v3} to the local ring $A_{(x)}$ together with the maps $\rho_x$ and $\partial_x$ in \eqref{eqCohHasse2p5p3v3}. Thus, we obtain
\[
\ker(\rho_x)\cap\im(\phi_x)=\ker(\rho_x)\cap \ker(\partial_x)=0\,.
\]Hence, the induced map $\rho_x: \im(\phi_x)\to \im(\tau_x)$ is injective.
This  injectivity combined with the the surjectivity of the map $g$ implies that  $\gamma_x$ is surjective, as can be shown by an easy diagram chase. This completes the proof of the theorem.
\end{proof}

\begin{remark}\label{remarkCohHasse2p6v3}
\

(1) The statement of the main theorem (Theorem$\;$\ref{thmCohHasse1p2v3}) does not hold if $r=1$. Counterexamples can be found in \cite[$\S$6]{CTPaSu} in the semi-global case and in \cite[Thm.$\;$1.5]{Ja} or \cite[Remark$\;$3.3]{CTOP} in the local case.

(2) In the proof of Theorem$\;$\ref{thmCohHasse2p5v3}, the only place where we have relied on the equi-characteristic assumption is the following version of Gersten's conjecture:
The complex
\[
0\lra H^r(A_{(x)}\,,\,\mathbb{Z}/n(r-1))\overset{\iota}{\lra} H^r(K_{(x)}\,,\,\mathbb{Z}/n(r-1))\lra\bigoplus_{v\in A_{(x)}^{(1)}}H^{r-1}(\kappa(v)\,,\,\mathbb{Z}/n(r-2))\,
\]
(with $A_{(x)}$ and $K_{(x)}$  defined as in Notation\;\ref{notationCohHasse2p4v3} (5)) is exact for every closed point $x$ of $\mathcal{X}$.

So, if Gersten's conjecture is true for the sheaf $\mathbb{Z}/n(r-1)$ over every 2-dimensional henselian excellent regular local ring, then Theorems$\;$\ref{thmCohHasse1p2v3} and \ref{thmCohHasse2p5v3} are still true without the equi-characteristic assumption.

(3) The Galois module $\mathbb{Z}/n(r-1)$ in Theorem$\;$\ref{thmCohHasse1p2v3} cannot be substituted by an arbitrary finite Galois module. In fact, in both the semi-global case and the local case, even with an algebraically closed residue field, there exists a finite Galois module $\mu$ such that the Hasse principle for $H^2(K,\,\mu)$ fails (\cite[Corollaires$\;$5.3 and 5.7]{CTPS16TAMS}).
\end{remark}

\section{Applications to torsors under semisimple groups}\label{sec3}

Unless otherwise stated, we will keep the notation and hypotheses of Theorem$\;$\ref{thmCohHasse1p2v3} in this section.
In particular, $A$ denotes an \emph{equi-characteristic}, henselian, excellent, normal local domain with residue field $k$. $\mathcal{X}$ is a two-dimensional regular integral scheme equipped with a proper morphism $\pi: \mathcal{X}\to\Spec(A)$ whose closed fiber is an snc divisor. $K$ is the function field of $\mathcal{X}$. In the semi-global case, $A$ is a discrete valuation ring and $\pi$ is flat. In the local case, $A$ is two-dimensional and $\pi$ is birational.

We shall give a number of Hasse principles for torsors under semisimple algebraic groups and for related structures. These Hasse principles are easy consequences of our cohomological Hasse principle (Theorem$\;$\ref{thmCohHasse1p2v3}) thanks to injectivity properties of certain cohomological invariants. Since this strategy is certainly well-known to experts and has already been used by several authors in similar contexts (see e.g. \cite{CTPaSu}, \cite{HHK12},  \cite{Preeti13}, \cite{Hu12}), here we will  only explain very briefly the results.

In the rest of this section, we denote by $G$ a semisimple simply connected algebraic group over $K$ and we consider the pointed set
\[
\sha(\mathcal{X}\,,\,G):=\ker\left(H^1(K,\,G)\lra \prod_{v\in \mathcal{X}^{(1)}}H^1(K_v,\,G)\right)\,.
\]When $G$ is absolutely simple, we denote by
  \[
  R_G\,:\; H^1(K,\,G)\lra H^3(K,\,\mathbb{Q}/\mathbb{Z}(2))
  \]the Rost invariant of $G$ over $K$. The order of $R_G$ can be determined explicitly according to the type of $G$ (cf. \cite[pp.437--442]{KMRT}). \emph{Here we will always assume that the order of $R_G$ is invertible in $k\,($or equivalently, in $K)$}. (In Theorems\;\ref{thmCohHasse3p2v3}, \ref{thmCohHasse3p4v3} and \ref{thmCohHasse3p5v3}, this condition boils down to the assumption $\mathrm{char}(K)\neq 2$.)

\vskip2mm

The following list of Hasse principles follows as were discussed in \cite[$\S$4.3]{HHK12} in the semi-global case. These results do not require any assumption on the residue field $k$ (except for the restriction on the characteristic with respect to the Rost invariant). We refer the reader to \cite{Preeti13} and \cite{Hu12} for more details when  $k$ is assumed finite.

\begin{itemize}
  \item (Groups of type ${}^1A_n^*$.) Let $D$ be a central simple $K$-algebra of square-free index and let $G=\mathbf{SL}_1(D)$.  The Rost invariant
     of the group $G$ is injective by \cite[Thm.$\;$24.4]{Suslin85}. (Our assumption on the order of the Rost invariant implies that the index of $D$ is invertible in $k$.) So we have $\sha(\mathcal{X},\,G)=1$ by Theorem$\;\ref{thmCohHasse1p2v3}$.

      \item (Groups of type $G_2$.) Let $G$ be the automorphism group $\mathbf{Aut}(C)$ of some Cayley algebra $C$ over $K$. If $\xi\in H^1(K,\,G)$ corresponds to a Cayley algebra $C'$, the Rost invariant $R_G$  maps $\xi$ to $e_3(N_C)-e_3(N_{C'})$, where $N_C$ and $N_{C'}$ denote the norm forms of $C$ and $C'$ respectively and $e_3$ is the Arason invariant for quadratic forms. Two Cayley algebras are isomorphic if and only if their norm forms are isomorphic. Since the norm form of a Cayley algebra is a 3-fold Pfister form and the Arason invariant is injective on Pfister forms by a well-known theorem of Merkurjev (cf. \cite[Prop.$\;$2]{Ara84}), the Rost invariant $R_G$ has trivial kernel. Hence, $\sha(\mathcal{X},\,G)=1$.

          A similar result was discussed in \cite[Example$\;$9.4 (a)]{HHK15AJM}.

  \item (Quasi-split groups.) For a quasi-split group $G$, we have
$\sha(\mathcal{X},\,G)=1$ in the each of the following cases:

  (1) $G$ is exceptional not of type $E_8$;

  (2) $G$ is of type $B_n$ with $2\le n\le 6$;

  (3) $G$ is of type $D_n$ with $3\le n\le 6$ or split of type $D_7$;

  (4) $G$ is of type ${}^2A_n$ with $n\le 5$.

Indeed, the triviality of the Rost kernel in these case has been proved in \cite{Gari01} (see also \cite{Chern03} in the first case).

  \item (Groups of type $F_4^{red}$.)   Recall that to each Albert algebra $J'$ over $K$ one can associate three cohomological invariants (cf. \cite[$\S$9.4]{SerreCohGalSemBour95} or \cite[$\S$40]{KMRT})
  \[
  f_3(J')\in H^3(K,\,\mathbb{Z}/2),\; f_5(J')\in H^5(K,\,\mathbb{Z}/2)\,\;\text{and }\quad g_3(J')\in H^3(K,\,\mathbb{Z}/3) \]
    such that the following properties hold:

    (1) $J'$ is reduced if and only if $g_3(J')=0$.

    (2) Two reduced Albert algebras are isomorphic if and only if they have the same  $f_3$ and $f_5$ invariants.

    Let $G=\mathbf{Aut}(J)$ be the automorphism group of a reduced Albert algebra $J$ over $K$.
    Our cohomological Hasse principle yields $\sha(\mathcal{X}\,,\,G)=1$ by means of the invariants $f_3,\,f_5$ and $g_3$.
\end{itemize}

We now discuss the Hasse principle for some other  groups.

\vskip2mm

For any field $F$ of characteristic different from 2, we denote by $W(F)$ its Witt group of quadratic forms and for each $r\ge 1$, we denote by $I^r(F)$ the subgroup generated by the classes of $r$-fold Pfister forms. By the quadratic form version of Milnor's conjecture (proved in \cite{OVV07}), there is a canonical homomorphism
\[
e_r\,:\; I^r(F)\lra H^r(F,\,\mathbb{Z}/2)
\]sending the class of a Pfister form
\[
\langle\!\langle a_1,\dotsc, a_r\rangle\! \rangle:=\langle 1,\,-a_1\rangle\otimes \cdots\otimes \langle 1,\,-a_r\rangle
\] to the cup product
\[
(a_1)\cup\cdots\cup (a_r)\;\in\; H^r(F,\,\mathbb{Z}/2)\,,
\]whose kernel is $I^{r+1}(F)$.

\begin{prop}\label{propCohHasse3p1v3}Let $\mathcal{X}$ and $K$ be as in Theorem$\;\ref{thmCohHasse1p2v3}$. Assume $\mathrm{char}(K)\neq 2$.

Then for each $r\ge 2$, the natural map
\[
I^r(K)\lra \prod_{v\in \mathcal{X}^{(1)}}I^r(K_v)
\]
is injective.
\end{prop}
\begin{proof}
Let $J_r$ be the kernel of the map in the proposition. The commutative diagram
\[\begin{CD}
  I^r(K) @>>> \prod I^r(K_v) \\
  @VVV @VVV\\
  H^r(K,\,\mathbb{Z}/2) @>>> \prod H^r(K_v,\,\mathbb{Z}/2)
\end{CD}\]
together with Theorem$\;$\ref{thmCohHasse1p2v3} shows that for every $r\ge 2$,
\[
J_r\subseteq I^{r+1}(K)=\ker\left(I^r(K)\lra H^r(K,\,\mathbb{Z}/2)\right)\,.
\]Therefore, we have for each $r\ge 2$,
\[
J_r=J_r\cap I^{r+1}(K)=J_{r+1}=J_{r+2}=\cdots =\bigcap_{d\ge r}J_d\,.
\]Since $\cap_{d\ge r}I^d(K)=0$ (cf. \cite[Coro.$\;$X.5.2]{Lam}), it follows immediately that $J_r=0$.
\end{proof}

In the semi-global case, the above local-global principle is also stated for $r=2$ in the last paragraph of \cite[$\S$9.2]{HHK15AJM}. (See also \cite[Lemma$\;$4.10]{Hu11} for the local case with finite residue field.)

\begin{thm}[Groups of type $C_n^*$]\label{thmCohHasse3p2v3}
Let $\mathcal{X}$ and $K$ be as in Theorem$\;\ref{thmCohHasse1p2v3}$. Assume $\mathrm{char}(K)\neq 2$.
Let $D$ be a quaternion division algebra over $K$ with standard
involution $\tau_0$ and $h$ a nonsingular hermitian form over
$(D,\,\tau_0)$. Let $G=\mathbf{U}(h)$ be the unitary group of
the hermitian form $h$.

Then the natural map
\[
H^1(K\,,\,G)\lra \prod_{v\in \mathcal{X}^{(1)}}H^1(K_v\,,\,G)
\]is injective. In particular,  one has $\sha(\mathcal{X},\,G)=1$.
\end{thm}
\begin{proof}
The pointed set $H^1(K\,,\,G)=H^1(K\,,\,\mathbf{U}(h))$ classifies up to
isomorphism hermitian forms over $(D,\,\tau_0)$ of the same rank as
$h$ (cf. \cite[Prop.$\;$2.16]{PR94}). Let $h_1$ and $h_2$ be two such hermitian forms. Put $h'=h_1\bot (-h_2)$ and let $q_{h'}$ be the trace form of $h'$. Two hermitian forms over
$(D\,,\,\tau_0)$ are isomorphic if and only if their trace forms are
isomorphic as quadratic forms (cf. \cite[p.352, Thm.$\;$10.1.7]{Schar}). So we need only to show that $[q_{h'}]=0$ in the Witt group $W(K)$. Note that the class  $[q_{h'}]\in W(K)$ lies in the subgroup $I^3(K)$. The theorem thus follows from Proposition$\;$\ref{propCohHasse3p1v3}.
\end{proof}

A special case of the following theorem has been discussed in \cite[Thm.$\;$4.9]{Hu11}.

\begin{thm}\label{thmCohHasse3p3v3}
Let $\mathcal{X}$ and $K$ be as in Theorem$\;\ref{thmCohHasse1p2v3}$. Assume $\mathrm{char}(K)\neq 2$.

Then for any nonsingular quadratic form $\phi$ of rank $\ge 2$ over
$K$, the natural map
\[
H^1(K\,,\,\mathbf{SO}(\phi))\lra\prod_{v\in\mathcal{X}^{(1)}}H^1(K_v\,,\,\mathbf{SO}(\phi))
\]
is injective.
\end{thm}
\begin{proof}
Let $\psi\,,\,\psi'$ be nonsingular quadratic forms representing classes in
$H^1(K\,,\,\mathbf{SO}(\phi))$. As they have the same dimension, the
forms $\psi$ and $\psi'$ are isometric if and only if they represent
the same class in the Witt group. Since $\psi$ and $\psi'$ also have the
same discriminant, we have $[\psi]-[\psi']\in I^2(K)$ by \cite[p.82, Chapt.$\;$2,
Lemma$\;$12.10]{Schar}. Then the theorem is immediate from Proposition$\;$\ref{propCohHasse3p1v3}.
\end{proof}

In the (semi-global) case where $A$ is a complete discrete valuation ring, one can use \cite[Thm.$\;$3.1]{CTPaSu}  to prove a variant of Theorem\;\ref{thmCohHasse3p3v3}  without the equi-characteristic assumption. Consequently, all the numbered results in this section have mixed characteristic versions in that case. Details of the argument will be given in Remark$\;$\ref{remarkCohHasse3p7v3}.

\begin{thm}[Special cases of groups of type ${}^2A_n$]\label{thmCohHasse3p4v3}
Let $\mathcal{X}$ and $K$ be as in Theorem$\;\ref{thmCohHasse1p2v3}$. Assume $\mathrm{char}(K)\neq 2$. Let $h$ be a nonsingular hermitian form of rank $\ge 2$ over a quadratic field extension $L$ of $K$ and let $G=\mathbf{SU}(h)$ be the special unitary group.

Then the natural map
\[
H^1(K\,,\,G)\lra \prod_{v\in \mathcal{X}^{(1)}}H^1(K_v\,,\,G)
\]is injective. In particular,  one has $\sha(\mathcal{X},\,G)=1$.
\end{thm}
\begin{proof}
The set $H^1(K,\,G)$ classifies nonsingular hermitian forms $h'$ which have the same rank and discriminant as $h$ (\cite[p.403, (29.19)]{KMRT}).
Let $q$ be the trace form of $h$. There is a natural embedding of $G$ into $\mathbf{SO}(q)$. The induced map
\[
H^1(K,\,G)=H^1(K,\,\mathbf{SU}(h))\lra H^1(K,\,\mathbf{SO}(q))
\]sending the class of a hermitian form $h'$ to the class of its trace form, is injective by \cite[Thm.$\;$10.1.1 (ii)]{Schar}. The result thus follows from Theorem$\;$\ref{thmCohHasse3p3v3}.
\end{proof}

\begin{thm}[Special cases of groups of type $B_n$ or $D_n$]\label{thmCohHasse3p5v3}
Let $\mathcal{X}$ and $K$ be as in Theorem$\;\ref{thmCohHasse1p2v3}$. Assume $\mathrm{char}(K)\neq 2$. Let $q$ be a nonsingular quadratic form of dimension $\ge 3$ over $K$ and let $G=\mathbf{Spin}(q)$.
If $q$ is isotropic over $K$, then $\sha(\mathcal{X},\,G)=1$.
\end{thm}
\begin{proof}
Consider the exact sequence of algebraic groups
\[
1\lra \mu_2\lra \mathbf{Spin}(q)\lra \mathbf{SO}(q)\lra 1
\]which gives rises to an exact sequence of pointed sets
\[
\mathbf{SO}(q)(K)\overset{\delta}{\lra}K^*/K^{*2}\lra H^1(K,\,\mathbf{Spin}(q))\lra H^1(K,\,\mathbf{SO}(q))\,.
\]Since $q$ is isotropic, the spinor norm map $\delta$ is surjective. It suffices to apply Theorem$\;$\ref{thmCohHasse3p3v3}, the Hasse principle for the group $\mathbf{SO}(q)$.
\end{proof}

Our Theorem$\;$\ref{thmCohHasse1p3v3} asserts that for a quasi-split semisimple simply connected group $G/K$ without $E_8$ factor, the Hasse principle for $G$-torsors holds.

\begin{proof}[Proof of Theorem$\;\ref{thmCohHasse1p3v3}$]By a standard argument using Shapiro's lemma (see e.g. \cite[$\S$3]{Hu12}), we may assume $G$ is absolutely simple. Choosing a proper morphism $\mathcal{X}\to \Spec(A)$ as in the introduction, we need only to show $\sha(\mathcal{X},\,G)=1$.

The cases ${}^1A_n$ and $C_n$ are trivial, since in these cases $H^1(F,\,G)=1$ for a quasi-split group $G$. The exceptional groups (not of type $E_8$) have been discussed earlier in this section. The cases of classical groups of type ${}^2A_n$, $B_n$ or $D_n$ are covered by Theorems$\;$\ref{thmCohHasse3p4v3} and \ref{thmCohHasse3p5v3}. This completes the proof.
\end{proof}

\begin{remark}\label{remarkCohHasse3p6v3}
There do exist quasi-split groups for which the Hasse principle holds but the Rost kernel is nontrivial. For example, the split group of type $B_7$ over $\mathbb{R}(\!(x,\,y)\!)$ has a nontrivial Rost kernel as shown in \cite[Example$\;$1.6]{Gari01}. But the Hasse principle for torsors under this group holds according to Theorem$\;$\ref{thmCohHasse1p3v3}.
\end{remark}

\begin{remark}\label{remarkCohHasse3p7v3}
If we consider discrete valuations arising from varying models $\mathcal{X}$, some results in this section extend to the mixed characteristic case when $A$ is a complete discrete valuation ring.

More precisely, let $F$ denote the fraction field of the complete discrete valuation ring $A$, and let $K$ be the function field of an algebraic curve over $F$. By the desingularization theory for two-dimensional schemes (\cite{Sha66}, \cite{Lip75}), there exists a proper morphism $\mathcal{X}\to\Spec(A)$ as in the semi-global case of Notation\;\ref{notationCohHasse1p1v3} such that $K$ is the function field of $\mathcal{X}$. By a \emph{divisorial valuation} of $K$ we mean a discrete valuation of $K$ that is defined by a codimension 1 point of such a scheme $\mathcal{X}$.

\emph{Assume that the residue field $k$ of $A$ has characteristic different form} 2. Then we claim that for the field $K$, the local-global statements in Proposition$\;$\ref{propCohHasse3p1v3} and Theorems\;\ref{thmCohHasse3p2v3} to \ref{thmCohHasse3p5v3} remain true if the set $\mathcal{X}^{(1)}$ is replaced with the set of divisorial valuations of $K$.

First of all, the variant of Theorem\;\ref{thmCohHasse3p3v3} is still true thanks to \cite[Thm.$\;$3.1]{CTPaSu}. This was explained in \cite[Remark$\;$4.12]{Hu11}. On the other hand, one can deduce Proposition\;\ref{propCohHasse3p1v3} from Theorem\;\ref{thmCohHasse3p3v3}.
Indeed, it is sufficient to show the injectivity of the local-global map for $I^2(K)$. Given a nonsingular quadratic form $\psi$ over $K$ whose class lies in $I^2(K)$, the dimension of $\psi$ is even and the discriminant of $\psi$ is trivial. If $\psi$ is locally hyperbolic, Theorem$\;$\ref{thmCohHasse3p3v3} applied to the hyperbolic form $\phi$ with $\dim\phi=\dim\psi$ implies that $\psi$ is hyperbolic over $K$.

The other three theorems follow from Proposition\;\ref{propCohHasse3p1v3} and Theorem\;\ref{thmCohHasse3p3v3} in the same way as before.

If $G$ is an absolutely simple simply connected group of type $G_2$ over $K$,  then the Hasse principle for $G$-torsors (with respect to the divisorial valuations) follows from the Hasse principle for $I^3(K)$, because Cayley algebras are classified by their norm forms.

Finally, if $G=\mathbf{SL}_1(D)$ for some quaternion algebra $D$ over $K$, the Hasse principle for $G$-torsors is also an easy consequence of  \cite[Thm.$\;$3.1]{CTPaSu}. This is because saying an element $a\in K^*$ is a reduced norm for $D$ is equivalent to saying that the quadratic form $\langle -a \rangle\bot N_D$ is isotropic, where $N_D$ denotes the norm form of $D$.

From the above, we get the following variant of Theorem\;\ref{thmCohHasse1p3v3}:

For the field $K$ as above, let $G$ be an absolutely simple simply connected group over $K$ such that the order of $R_G$ is invertible in the residue field $k$. Then the Hasse principle for $G$-torsors with respect to the divisorial valuations of $K$ holds if $G$ is of type $G_2$ or quasi-split of classical type.
\end{remark}

\section{Applications to Pythagoras numbers}\label{sec4}
In this section we prove Theorem$\;\ref{thmCohHasse1p4v3}$. The basic strategy is essentially the same as in the proof of \cite[Thm.$\;$5.1]{Hu13}.

\

\newpara\label{paraCohHasse4p1v3} Recall some general facts on the Pythagoras number of fields.

Let $F$ be a field of characteristic $\neq 2$. Recall that the level $s(F)$ of $F$ is the smallest integer $s\ge 1$ or $+\infty$, such that $-1$ is the sum of $s$ squares in $F$. If $s(F)=+\infty$, the field $F$ is called (formally) real. Otherwise it is called nonreal. For any field $F$ one has $s(F(t))=s(F)=s(F(\!(t)\!))$.

If $F$ is a nonreal field, $s(F)$ is a power of 2 and one has
\begin{equation}\label{eqCohHasse4p1p1v3}
s(F)\le p(F)\le p(F(x))=s(F)+1=s(F(x))+1=p(F(x,\,y))\,.
\end{equation}

The following theorem will be  referred to as \emph{Pfister's theorem}.
\begin{thm}[Pfister]\label{thmCohHasse4p2v3}
Let $F$ be a real field and $r\ge 1$ an integer. Then the following two conditions are equivalent:

$(\mathrm{i})$ $p(F(t))\le 2^r$.

$(\mathrm{ii})$ $s(L)\le 2^{r-1}$ for every finite nonreal extension $L/F$.
\end{thm}
\begin{proof}See e.g. \cite[p.397, Examples$\;$XI.5.9 (3)]{Lam}.
\end{proof}

\begin{remark}\label{remarkCohHasse4p3v3}
Becher and Van Geel have shown in \cite[Thm.$\;$3.5]{BvG09} that the two conditions in Pfister's theorem are also equivalent to the following:

(iii) $p(L)<2^r$ for any finite real extension $L/F$.
\end{remark}

Given a field $L$ of characteristic $\neq 2$ and an integer $r\ge 1$, the following assertions are equivalent:

(1) $s(L)\le 2^{r-1}$.

(2) The $r$-fold Pfister form $\langle\!\langle -1,\dotsc, -1\rangle\!\rangle=\langle 1,\,1\rangle^{\otimes r}$ is isotropic over $L$.

(3) The $r$-fold Pfister form $\langle\!\langle -1,\dotsc, -1\rangle\!\rangle=\langle 1,\,1\rangle^{\otimes r}$ is hyperbolic over $L$.

(4) The cohomology class $(-1)\cup\cdots\cup (-1)\in H^r(L,\,\mathbb{Z}/2)$ vanishes.

Here the equivalences of (1), (2) and (3) are classical and that they are equivalent to (4) is guaranteed by the quadratic form version of Milnor's conjecture  (proved in \cite{OVV07}).

\begin{remark}\label{remarkCohHasse4p4v3}
What we really need in this paper is only the equivalence between properties (3) and (4) above, but not the full strength of Milnor's conjecture. While a $K$-theoretic analog (cf. \cite[Thm.$\;$X.6.7]{Lam}) was known many years ago, this fact seems to have been proved only for $r\le 4$ (cf. \cite{JacobRost89} and \cite{MS90}) prior to the solution of Milnor's conjecture.
\end{remark}

Now consider a Laurent series field $F_n=k(\!(t_1,\dotsc, t_n)\!)$ with $n\ge 3$. A theorem of Pfister (cf. \cite[p.90, Thm.$\;$3.3]{Pfister95}) implies that for any integer $r\ge 1$, if $F$ is a field such that every $(r+1)$-fold Pfister form over $F(\sqrt{-1})$ is hyperbolic, then $p(F)\le 2^r$. Combined with Milnor's conjecture, this shows that if the  cohomological $2$-dimension $\mathrm{cd}_2(F(\sqrt{-1}))$ of the field $F(\sqrt{-1})$ is at most $r$, then $p(F)\le 2^r$.

If $k$ is a field of characteristic 0 such that $\mathrm{cd}_2(k(\sqrt{-1}))=r$, we have
$\mathrm{cd}_2(F_n(\sqrt{-1}))\le n+r$ by \cite[XIX, Coro.$\;$6.3]{SGA4iii}. Hence
$p(k(\!(t_1,\dotsc, t_n)\!))\le 2^{n+r}$ for such a field $k$. We call this estimate \emph{Pfister's bound}.

Note that Pfister's bound is not optimal already in the case $k=\mathbb{R}$.  Our goal here is to give a better estimate in the case $n=3$.

\

We are now ready to prove Theorem$\;$\ref{thmCohHasse1p4v3}, which states that if $k$ is a field such that $p(k(x,\,y))\le 2^r$, then
\[
p(k(\!(x,\,y,\,z)\!))\le p(k(\!(x,\,y)\!)(t))\le p(k(\!(x,\,y)\!)(\!(z,\,t)\!))\le 2^r\,.
\]
Here equalities hold when $p(k(x,\,y))=2^r$, since $p(k(x,\,y))\le p(k(\!(x,\,y,\,z)\!))$ by \cite[Coro.$\;$2.3]{Hu13}.

\begin{proof}[Proof of Theorem$\;\ref{thmCohHasse1p4v3}$]
If $k$ is nonreal, then by \eqref{eqCohHasse4p1p1v3} and \cite[Prop.$\;$3.2]{Hu13} we have
\[
p(k(\!(x,\,y,\,z)\!))=p(k(\!(x,\,y)\!)(t))=s(k)+1=s(k(x))+1=p(k(x,\,y))\,
\]
and
\[
p(k(\!(x,\,y)\!)(\!(z,\,t)\!))=s(k(\!(x,\,y)\!))+1=s(k)+1=p(k(x,\,y))\,.
\]
So we may assume $k$ is real. It has been proved in \cite[Corollaries$\;$2.3 and 4.4]{Hu13} that the inequalities
\[
p(k(x,\,y))\le p(k(\!(x,\,y,\,z)\!))\le p(k(\!(x,\,y)\!)(t))\,
\]hold in general. Moreover,
\[
p(k(\!(x,\,y)\!)(t))\le p(k(\!(x,\,y)\!)(\!(z,\,t)\!))\,,
\]and the two Pythagoras numbers in this inequality are bounded by the same 2-powers whenever such bounds exist (cf. \cite[Thm.$\;$5.18]{CDLR} and \cite[Thm.$\;$3.1]{Hu13}).

It is thus sufficient to show the inequality $p(k(\!(x,\,y)\!)(t))\le 2^r$ under the assumption $p(k(x,\,y))\le 2^r$. By Pfister's theorem, we need only to show $s(K)\le 2^{r-1}$ for every finite nonreal extension $K$ of the field $k(\!(x,\,y)\!)$. Since this property has an interpretation in terms of the vanishing of a cohomology class in $H^r(K,\,\mathbb{Z}/2)$ and we have a Hasse principle for the cohomology group $H^r(K,\,\mathbb{Z}/2)$ (Theorem$\;\ref{thmCohHasse1p2v3}$), it suffices to show that if $A$ is the integral closure of $k[\![x,\,y]\!]$ in $K$ and if $\mathcal{X}\to\Spec(A)$ is chosen as in Notation\;\ref{notationCohHasse1p1v3}, then $s(K_v)\le 2^{r-1}$ for every $v\in \mathcal{X}^{(1)}$.

For each $v\in\mathcal{X}^{(1)}$, the residue field $\kappa(v)$ has the same level as $K_v$, and $\kappa(v)$ is either  isomorphic to $k'(\!(t)\!)$ for some finite nonreal extension $k'/k$, or a function field of transcendence degree 1 over $k$. In the former case, Pfister's theorem applied to the real field $k$ shows that $s(\kappa(v))=s(k')\le 2^{r-1}$, since $p(k(t))\le p(k(x,\,y))\le 2^r$. In the latter case, $\kappa(v)$ is (isomorphic to) a finite nonreal extension of $k(x)$. Again by Pfister's theorem, applied to the real field $k(x)$ this time, we get $s(\kappa(v))\le 2^{r-1}$. The theorem is thus proved.
\end{proof}

Theorem$\;$\ref{thmCohHasse1p4v3} allows us to generalize the results in \cite[$\S$5]{Hu13}.

\begin{coro}\label{corCohHasse4p5v3}
Let $k$ be an algebraic function field of transcendence degree $d\ge 0$ over a field $k_0$.

$(\mathrm{i})$ If $k_0$ is a real closed field, then \[
p(k(\!(x,\,y,\,z)\!))\le p(k(\!(x,\,y)\!)(t))\le p(k(\!(x,\,y)\!)(\!(z,\,t)\!))\le 2^{d+2}\,.
\]

$(\mathrm{ii})$ If $k_0$ is a number field, i.e., a finite extension of $\mathbb{Q}$, then
\[p(k(\!(x,\,y,\,z)\!))\le p(k(\!(x,\,y)\!)(t))\le p(k(\!(x,\,y)\!)(\!(z,\,t)\!))\le 2^{d+3}\,.\]
\end{coro}
\begin{proof}
In case (i) we have $p(k(x,\,y))\le 2^{d+2}$ by \cite[p.95, Examples$\;$1.4 (4)]{Pfister95}.

In case (ii), thanks to Jannsen's work on Kato's conjecture \cite{Jannsen16} and the proof of Milnor's conjecture (\cite{OVV07}), we can deduce $p(k(x,\,y))\le 2^{d+3}$ from \cite[Thm.$\;$4.1]{CTJannsen91}.   So the corollary follows from Theorem$\;$\ref{thmCohHasse1p4v3}.
\end{proof}

In the above corollary, Pfister's bound  is $2^{d+3}$ in case (i) and $2^{d+5}$ in case (ii). The upper bound in (i) is likely to be optimal when $k$ is a rational function field.

\begin{lemma}\label{lemmaCohHasse4p6v3}
  Let $k_0$ be a field of characteristic $\neq 2$, $n\ge 1$ and $k=k_0(\!(t_1)\!)\cdots (\!(t_n)\!)$. Let $r\ge 2$ be an integer such that $p(k_0(x,\,y))\le 2^r$. Then $p(k(x,\,y))\le 2^r$.
\end{lemma}
\begin{proof}
By induction we may reduce to the case $n=1$. If $k_0$ is nonreal,  the result follows from the fact that $s(k_0)=s(k_0(\!(t)\!))=s(k)$ in view of \eqref{eqCohHasse4p1p1v3}. So we may assume $k_0$ is real. By Pfister's theorem, it suffices to prove $s(K)\le 2^{r-1}$ for every finite nonreal extension $K$ of $k(x)=k_0(\!(t)\!)(x)$. The argument for a similar statement given in our proof of Theorem$\;$\ref{thmCohHasse1p4v3} works verbatim, since our cohomological Hasse principle is also valid in the semi-global case. Alternatively, one can use \cite[Thm.$\;$6.7]{BGvG12}.
\end{proof}

\begin{coro}\label{coroCohHasse4p7v3}
 Let $k_0$ be a field, $n\ge 1$ and $k=k_0(\!(t_1)\!)\cdots (\!(t_n)\!)$.

$(\mathrm{i})$ If $k_0$ is a function field of transcendence degree $d\ge 0$ over a real closed field, then \[
p(k(\!(x,\,y,\,z)\!))\le p(k(\!(x,\,y)\!)(t))\le p(k(\!(x,\,y)\!)(\!(z,\,t)\!))\le 2^{d+2}\,.
\]

$(\mathrm{ii})$ If $k_0$ is a function field of transcendence degree $d\ge 0$ over $\mathbb{Q}$, then
\[p(k(\!(x,\,y,\,z)\!))\le p(k(\!(x,\,y)\!)(t))\le p(k(\!(x,\,y)\!)(\!(z,\,t)\!))\le 2^{d+3}\,.\]
\end{coro}
\begin{proof}
We get the desired upper bound for $p(k_0(x,\,y))$ as in the proof of Corollary\;\ref{corCohHasse4p5v3}.  By Lemma$\;$\ref{lemmaCohHasse4p6v3}, this yields the same bound for $p(k(x,\,y))$, so the result follows from Theorem$\;$\ref{thmCohHasse1p4v3}.
\end{proof}

As mentioned in the introduction, the results of the above corollary improve greatly on Pfister's bounds.

\

\noindent \emph{Acknowledgements.} The author is grateful to an anonymous referee for helpful comments.

\addcontentsline{toc}{section}{\textbf{References}}

\bibliographystyle{alpha}

\bibliography{Hasse}

\

Author information:

\

Yong HU

\

Universit\'e de Caen, Campus 2

Laboratoire de Math\'ematiques Nicolas Oresme

14032, Caen Cedex

France

Email: yong.hu@unicaen.fr

\end{document}